\documentclass{amsart}
\usepackage{amsmath,amssymb,amsthm,
verbatim}
\newtheorem{theorem}{Theorem}
\newtheorem{proposition}{Proposition}
\newtheorem{lemma}{Lemma}

\theoremstyle{definition}
\newtheorem{definition}{Definition}

\begin{document}
\title{On a Question of Glasby, Praeger, and Xia}
\author{Michael~J.~J.~Barry}
\address{Department of Mathematics\\
Allegheny College\\
Meadville, PA 16335}
\email{mbarry@allegheny.edu}
\thanks{This work was done when the author was on sabbatical and he thanks Allegheny College for its support.}


\begin{abstract}
Recently, Glasby, Praeger, and Xia asked for necessary and sufficient conditions for the `Jordan Partition' $\lambda(r,s,p)$ to be standard.  We give such conditions when $p$ is an odd prime.
\end{abstract}
\maketitle

\section{Introduction}\label{S:intro}
As usual $p$ is a prime number. There are different ways to explain the notion of Jordan Partition and we approach it via the modular representations of a finite cyclic $p$-group $G$ of order $q=p^\alpha$ over a field $K$ of characteristic $p$.  It is well-known that there are exactly $q$ isomorphism classes of indecomposable $KG$-modules.  Let $\{V_1,\dots,V_q\}$ be a set of representatives of these isomorphism classes with $\dim V_i=i$.  Many authors have investigated the decomposition of the $K G$-module $V_m \otimes V_n$, where $m \leq n$, into a direct sum of indecomposable $K G$-modules --- for example, in order of publication, see \cite{G1962}, \cite{S1964}, \cite{L1974},  \cite{N1995}, \cite{R1979}, \cite{H2003},  and \cite{B2011}. From the works of these authors, it is well-known that $V_m \otimes V_n$ decomposes into a direct sum $V_{\lambda_1} \oplus \dots \oplus V_{\lambda_m}$ of $m$ indecomposable $KG$-modules with $\lambda_1 \geq \dots \geq \lambda_m>0$, but that the dimensions $\lambda_i$ of the components depend on the characteristic $p$.  Following~\cite{GPX2014}, we define the {\bf Jordan Partition} $\lambda(m,n,p)$ of $m n$ by
\[\lambda(m,n,p)=(\lambda_1,\dots,\lambda_m).\]
We say that $\lambda(m,n,p)$ is {\bf standard} if{f} $\lambda_i=m+n-2 i+1$ for $1 \leq i \leq m$.

A sufficient reason for $\lambda(m,n,p)$ to be standard was given in~\cite[Theorem 2]{GPX2014}, and Problem 16 of the same paper asked for necessary and sufficient conditions.  We give these conditions now when $p$ is odd in the following two theorems which deal with the cases $m<p$ and $m \geq p$, respectively.

\begin{theorem}\label{Theorem2}
Assume that $p$ is odd.  Define $S=S_1 \cup S_2$, where $S_1=\{(k,d) \mid 1 \leq k \leq d \leq p+1-k\}$ and
\[
S_2=\{(k,b p+d) \mid b \geq 1, 1 \leq k \leq (p+1)/2,k-1 \leq d \leq p+1-k \}.
\]
If $1 \leq m<p$ and $n \geq m$, then $\lambda(m,n,p)$ is standard if{f} $(m,n) \in S$.
\end{theorem}

\begin{theorem}\label{Theorem3}
Assume that $p$ is odd.  Define 

\[S=\{(i p^t+(p^t \pm 1)/2,j p^t+(p^t \pm 1)/2 +k p^{t+1}) \mid 1 \leq i \leq (p-1)/2, i \leq j \leq p-i-1, k \geq 0\}.\]
Suppose that $p^t \leq m <p^{t+1}$ with $t \geq 1$ and $n \geq m$.  Then $\lambda(m,n,p)$ is standard if{f} $(m,n) \in S$.
\end{theorem}

In~\cite{B2011}, we gave a recursive definition of the combinatorial object $s_p(m,n)$ and proved that
\[s_p(m,n)=(\lambda_1,\dots,\lambda_m, \underbrace{0,\dots,0}_{n-m},-\lambda_m,\dots,-\lambda_1).\] 
In Section~\ref{S:$s_p$}, we will define $s_p(m,n)$, which will be the main tool in our proofs of Theorems~\ref{Theorem2}~and~\ref{Theorem3} in Section~\ref{S:Proofs}.

\section{Definition of $s_p(m,n)$}\label{S:$s_p$}
Assume that $m$ and $n$ are positive integers with $m \leq n$.  Before we give a formal recursive definition, let us say that $s_p(m,n)$ is a nonincreasing sequence of $m+n$ integers whose first $m$ terms are positive, whose last $m$ terms are negative, and whose middle $n-m$ terms all equal 0.  Further, letting $s_p(m,n)(k)$ denote the $k$th term of $s_p(m,n)$, the sequence is ``balanced around its middle''  in the sense that \[s_p(m,n)(m+n+1-k)=-s_p(m,n)(k), \qquad k=1,\dots,m+n,\]
and its positive terms sum to $m \cdot n$---so $\sum_{k=1}^m s_p(m,n)(k)=m \cdot n$.  For example, 
\[s_5(6,7)=(12,10,8,5,5,2,0,-2,-5,-5,-8,-10,-12)\]
and 
\[s_3(6,8)=(9,9,9,9,9,3,0,0,-3,-9,-9,-9,-9,-9). \]
The positive terms in $s_p(m,n)$ will turn out to be the dimensions of the indecomposable modules in the decomposition of $V_m \otimes V_n$.

We begin by explaining our notation.  All our sequences are finite nonincreasing sequences of integers.  If $s=(a_1,\dots,a_u)$ and $t=(b_1,\dots,b_v)$ are two sequences with $a_u \geq b_1$, then the sequence $s \oplus t$ is defined by
\[s \oplus t=(a_1,\dots,a_u,b_1,\dots,b_v),\]
the concatenation of the two sequences.  Following~\cite{GPX2014}, the {\bf negative reverse} $\overline{s}$ of $s$ is defined by $\overline{s}=(-a_u,\dots,-a_2,-a_1)$.  For an integer $m$ and a positive integer $k$, $(m:k)$ denotes the sequence
\[(\underbrace{m,\dots,m}_k).\] 
We will also denote the empty sequence $()$ by $(0:0)$.  If $s$ is a sequence, then $s_>$ and $s_<$, respectively, denote the subsequences of $s$ consisting of all positive terms, and all negative terms, respectively.  For example,
\[s_3(6,8)_{<}=(-3,-9,-9,-9,-9,-9). \]

For a sequence $s$ and an integer $k$,  $s+k$ denotes the sequence obtained from $s$ by adding $k$ to each of its terms. For example,
\[s_3(6,8)_{<}+2 \cdot 3^2=(15,9,9,9,9,9).\]

We now define $s_p(m,n)$ which was introduced in~\cite{B2011}.


\begin{definition}\label{D:one}
Let $p$ be a prime and let $m$ and $n$ be integers satisfying $0 \leq m \leq n$.  Define $s_p(0,n)=(0:n)$.  Assume now that $0<m \leq n$ and let $k$ be the unique nonnegative integer such that $p^k \leq n<p^{k+1}$.    Write $n= b p^k+d$ with $0<b<p$ and $0 \leq d <p^k$.  Write $m=a p^k +c$ with $0 \leq a<p$ and $0 \leq c <p^k$.  Note that $a+c>0$. We define $s_p(m,n)$ recursively as 
\[s_p(m,n)=s_1 \oplus s_2 \oplus s_3,\]
where $s_3=\overline{s_1}$ and $s_1$ and $s_2$ are given in the following exhaustive list of cases.
\begin{enumerate}
\item  Case 1: $m+n>p^{k+1}$.  Then 
$s_1=(p^{k+1}:m+n-p^{k+1})$ and $s_2=s_p(p^{k+1}-n,p^{k+1}-m)$. \label{Case1}
\item Case 2: $m+n \leq p^{k+1}$ and $c+d>p^k$.  Then
$s_1=((a+b+1)p^k:c+d-p^k)$ and $s_2=s_p((a+b+1)p^k-n,(a+b+1)p^k-m)$. \label{Case2}
\item Case 3: $m+n \leq p^{k+1}$, $1 \leq c+d \leq p^k$, and $a>0$.  Then
$s_1=s_p(\min(c,d),\max(c,d))+(a+b)p^k$ and $s_2=s_p((a+b)p^k-n,(a+b)p^k-m)$. \label{Case3}
\item Case 4: $m+n \leq p^{k+1}$, $1 \leq c+d \leq p^k$, $a=0$ (so $m=c$), and $d>0$.  Then
$s_1=s_p(m,b p^k-d)_<+2 b p^k$ and $s_2=(0:n-m)$. \label{Case4}
\item Case 5: $m+n \leq p^{k+1}$, $1 \leq c+d \leq p^k$, $a=0$, and $d=0$, so $(m,n)=(c,b p^k)$.  Then
$s_1=(b p^k:m)$ and $s_2=(0:b p^k-m)$. \label{Case5}
\item Case 6: $m+n \leq p^{k+1}$, $c=d=0$, so $0<a<a+b \leq p$.  Then
$s_1=((a+b-1)p^k:p^k)$ and $s_2=s_p((a-1)p^k,(b-1)p^k)$. \label{Case6}
\end{enumerate}
\end{definition} 

In Case~\ref{Case6}, one can show easily that
\[s_1=((a+b-1)p^k:p^k)\oplus ((a+b-3)p^k:p^k)\oplus \dots \oplus ((b-a+1)p^k:p^k)\] and $s_2=(0:(b-a)p^k)$.
When $k=0$, so $(m,n)=(a,b)$, this specializes to $s_1=(a+b-1,a+b-3, \dots, b-a+1)$ and $s_2=(0:b-a)$.

Recall that for a sequence $s$ and an integer $k$, $s(k)$ denotes the $k$th term of the sequence $s$.  The following result was proved in~\cite{B2011}.

\begin{theorem}\label{T:one}
For positive integers $m$ and $n$ with $m \leq n \leq q$, 
\[V_m \otimes V_n= \bigoplus_{k=1}^{m} V_{s_p(m,n)(k)}.\]
\end{theorem}

It follows, as we had stated previously, that
\[\lambda(m,n,p)=(s_p(m,n)(1),s_p(m,n)(2),\dots,s_p(m,n)(m)).\] 

In the next result we characterize exactly when $\lambda(m,n,p)$ is standard for each of the six cases of Definition~\ref{D:one}.

\begin{proposition}\label{P:one}
The Jordan partition $\lambda(m,n,p)$ is standard if{f}
\begin{enumerate}
\item $m+n-p^{k+1}=1$ and $\lambda(p^{k+1}-n,p^{k+1}-m,p)$ is standard in Case~\ref{Case1}
\item $c+d-p^k=1$ and $\lambda((a+b+1)p^k-n,(a+b+1)p^k-m,p)$ is standard in Case~\ref{Case2}
\item $\lambda(\min(c,d),\max(c,d),p)$ is standard, $|c-d| \leq 1$ and $\lambda((a+b)p^k-n,(a+b)p^k-m,p)$ is standard in Case~\ref{Case3}
\item $\lambda(m,b p^k-d,p)$ is standard in Case~\ref{Case4}
\item $m=1$ in Case~\ref{Case5}
\item $k=0$ in Case~\ref{Case6}
\end{enumerate}
\end{proposition}
\begin{proof}
All except Case~\ref{Case3} are completely obvious.  In this case, we note that since $s_p(\min(c,d),\max(c,d))$ and not just $\lambda(\min(c,d),\max(c,d),p)$ is involved in $\lambda(m,n,p)$, if $|c-d|>1$, then  $s_p(\min(c,d),\max(c,d))$ has repeated $0$'s and so $\lambda(m,n,p)$ has repeated dimensions.
\end{proof}

\section{Proofs}\label{S:Proofs}
First we assemble some lemmas beginning with a special case of Theorem~\ref{Theorem2}.
\begin{lemma}\label{Lemma1}
If $1 \leq m,n<p$, then $\lambda(m,n,p)$ is standard if{f} $m+n \leq p+1$.
\end{lemma}
\begin{proof}
We can assume that $1 <m \leq n$.  In terms of Definition~\ref{D:one}, $k=0$, $c=d=0$, $m=a$, and $n=b$, leaving us in either Case~\ref{Case1} or Case~\ref{Case6}.  By the remarks on Case~\ref{Case6} after Definition~\ref{D:one}, $\lambda(m,n,p)$ is standard in Case~\ref{Case6}.  By Proposition~\ref{P:one}, $\lambda(m,n,p)$ is standard in Case~\ref{Case1} if{f} $m+n=p+1$ and $ \lambda(p-n,p-n,p)$ is standard.  But if $m+n=p+1$, $ \lambda(p-n,p-n,p)=\lambda(m-1,n-1,p)$ where $(m-1)+(n-1)=p-1$, that is, Case~\ref{Case6}, hence standard.
\end{proof}

\begin{lemma}\label{Lemma2} Here $t$ is a positive integer,  $x=(p^t \pm 1)/2$, and $y=(p^t \pm 1)/2$.
\begin{enumerate}
\item Suppose $x \leq y$.  Then for all integers $i$ in the interval $[0,(p-1)/2]$, $\lambda(i p^t+x, i p^t+y,p)$ is standard. \label{Lemma2_1}
\item For any integer $b$, if $1 \leq b \leq p-1$, then $\lambda(x,b p^t+ y,p)$ is standard.\label{Lemma2_2}
\end{enumerate}
\end{lemma}

\begin{proof}
\eqref{Lemma2_1}  By contradiction.  Let $t$ be the least positive integer for which this is false.  For this $t$, let $i$ be the least integer for which it is false.  Next we show $i>0$.  If $t=1$, then $i$ cannot be $0$ by Lemma~\ref{Lemma1}.  If $t>1$, $x=\frac{p^t \pm 1}{2}=\frac{p-1}{2} p^{t-1}+\frac{p^{t-1} \pm 1}{2}$ and $y$ has a similar expression.  Since the result is true for $t-1$, $\lambda(x, y,p)$ is standard.  We have shown that $i>0$.

Only the first three cases of Definition~~\ref{D:one} apply but we consider Case~\ref{Case3} first because the other two reduce to this.

Case~\ref{Case3}: First $|x-y| \leq 1$.  We have just seen that $\lambda(\min(x,y),\max(x,y),p)$ is standard.  Since $((i-1)p^t +p^t-\max(x,y),(i-1)p^t+p^t-\min(x,y))=((i-1)p^t+x',(i-1)p^t+y')$ where $x'=(p^t \pm 1)/2$ and $y'=(p^t \pm 1)/2$, $\lambda((i-1)p^t+x',(i-1)p^t+y',p)$ is standard by assumption.  Hence  $\lambda(i p^t+x, i p^t+y,p)$ is standard by Proposition~\ref{P:one}.

Case~\ref{Case1}: Here $(i p^t+x)+ (i p^t+y)>p^{t+1}+1$.  The only possibility is $i=(p-1)/2$ and $x=(p^t+1)/2=y$.  Now $(p^{t+1}-(p-1)/2 \cdot p^t-(p^t+1)/2,p^{t+1}-(p-1)/2 \cdot p^t-(p^t+1)/2)=((p-1)/2 \cdot p^t +x',(p-1)/2 \cdot p^t +y')$ where $x'=(p^t-1)/2=y'$.  This is a Case~\ref{Case3} situation.  Note that $|x'-y'|=0$ and $\lambda(x',y',p)$ is standard.  In addition, $(p-1)p^t-(p-1)/2 \cdot p^t -x'=(p-3)/2 \cdot p^t +p^t-x'=(p-3)/2 \cdot p^t+x$.  Hence $\lambda((p-3)/2 \cdot p^t+x,(p-3)/2 \cdot  p^t+x,p)$ is standard since $(p-3)/2<i=(p-1)/2$.  By Proposition~\ref{P:one}, $\lambda((p-1)/2 \cdot p^t+(p^t+1)/2,(p-1)/2 \cdot p^t+(p^t+1)/2,p)$ is standard.

The treatment of Case~\ref{Case2} is similar to the treatment of Case~\ref{Case1}.

\eqref{Lemma2_2}  By contradiction.  Let $b$ the least such integer for which $\lambda(x,b p^t+ y,p)$ is not standard.  We consider the relevant cases.

Case~\ref{Case4}: $x+b p^t+y \leq p^{t+1}$ and $1 \leq x+y \leq p^t$.  Then either $x=y=(p-1)/2$ or exactly one of $x$ and $y$ equals $(p-1)/2$ while the other equals $(p+1)/2$.  Then since $(x,b p^t-y)=(x,(b-1)p^t+p^t-y)$, $\lambda(x,b p^t-y,p)$ is standard by the definition of $b$ if $b>1$ or by Part~\ref{Lemma2_1} if $b=1$.  It follows that  $\lambda(x,b p^t+y,p)$ is standard.

Case~\ref{Case2}: $x+b p^t+y \leq p^{t+1}$ but $x+y>p^t$.  Here $b \leq p-2$.  It must be that $x=y=(p+1)/2$.  Since $\lambda(x,b p^t+ y,p)$ is not standard, $\lambda((1+b)p^t-b p^t-y, (1+b)p^t-x,p)$, that is,  $\lambda(p^t-y,b p^t+p^t-x,p)$ is not standard.  This is a Case~\ref{Case4} situation since $(p^t-y)+(p^t-x)<p^t$.  Hence $\lambda(p^t-y, b p^t-p^t+x,p)=\lambda(p^t-y, (b-1)p^t+x,p)$ is not standard.  This contradicts our choice of $b$ if $b>1$ and contradicts Part~\ref{Lemma2_1} of Lemma~\ref{Lemma2} if $b=1$.

Case~\ref{Case1} is handled similarly to Case~\ref{Case2}.
\end{proof}

The next two results are just special cases of Proposition 3 of~\cite{GPX2014} but we will prove them in the setting of Definition~\ref{D:one}.

\begin{lemma}\label{Lemma3}
If $m<p^t$ and $p^t$ divides $n$, then $\lambda(m,n,p)=(n,\dots,n)$.
\end{lemma}

\begin{proof}
Write $n=f p^t$ and proceed by induction on $f$.  When $f=1$, $m=0 \cdot p^t+m$ and $p^t=1 \cdot p^t +0$,  so we are in Case~\ref{Case5}, and $\lambda(m,p^t,p)=(p^t,\dots,p^t)$.  Hence the result holds  when $f=1$.  Now let $f \geq 2$ and assume that the result holds for all integers less than $f$.  Write $f p^t=b p^k+d$ where $0<b<p$ and $0 \leq d <p^k$.  Note that $t \leq k$ and $p^t$ divides $d$.  Here $m=0 \cdot p^k+m$.  We are either in Case~\ref{Case4} if $d>0$ or Case~\ref{Case5} if $d=0$.  In Case~\ref{Case5}, 
\[\lambda(m,f p^t,p)=(f p^t,\dots,f p^t)=(n, \dots,.n).\] 
In Case~\ref{Case4}, $m+d \leq p^k$ and
\[\lambda(m,b p^k+d,p)=s_p(m,b p^k-d)_<+2 b p^k.\]
Since $p^t$ divides $b p^k-d$, $\lambda(m,b p^k-d,p)=(b p^k-d,\dots,b p^k-d)$.  Hence
$s_p(p^t,b p^k-d)_<=(-b p^k+d, \dots, -b p^k+d)$ and
$\lambda(m,b p^k+d,p)=(n, \dots,n)$.
\end{proof}

\begin{lemma}\label{Lemma4}
If $t$ and $n$ are positive integers with $p^t \leq n$, then $\lambda(p^t,n,p)$ is not standard.
\end{lemma}

\begin{proof}
By contradiction.  Suppose that $n$ is the least integer $\geq p^t$ for which $\lambda(p^t,n,p)$ is standard.  Write $n =b p^k+d$ where $1 <b<p$ and $0 \leq d <p^k$, and write $p^t=a p^k+c$ where $0 \leq a<p$ and $0 \leq d <p^k$.  If $t<k$, then $a=0$ and $c=p^t$; if $t=k$, then $a=1$ and $c=0$.  We consider the six cases of Definition~\ref{D:one}.

Case~\ref{Case1}: $p^t+n>p^{k+1}$.  If $p^t+n>p^{k+1}+1$, then $\lambda(p^t,n,p)$ is not standard.  We can assume that $p^t+n=p^{k+1}+1$.  In order for $\lambda(p^t,n,p)$ to be standard in this case, $\lambda(p^{k+1}-n,p^{k+1}-p^t,p)=\lambda(p^t-1,p^{k+1}-p^t,p)$ must be standard. But by Lemma~\ref{Lemma3}, $\lambda(p^t-1,p^{k+1}-p^t,p)=(p^{k+1}-p^t,\dots, ,p^{k+1}-p^t)$ and so is not standard, implying that $\lambda(p^t,n,p)$ is not standard.

Case~\ref{Case2}: $p^t+n \leq p^{k+1}$ but $c+d > p^k$.  If $c+d> p^k+1$, $\lambda(p^t,n,p)$ is not standard.  We can assume that $c+d=p^{k}+1$.  In order for $\lambda(p^t,n,p)$ to be standard in this case, $\lambda((a+b+1)p^k-n,(a+b+1)p^k-p^t,p)=\lambda(p^t-1,(a+b+1)p^k-p^t,p)=\lambda(p^t-1,n-1,p)$ must be standard.  But by Lemma~\ref{Lemma3}, $\lambda(p^t-1,(a+b+1)p^k-p^t,p)=(n-1,\dots,n-1)$ and so is not standard, implying that $\lambda(p^t,n,p)$ is not standard.

Case~\ref{Case3}: $m+n \leq p^{k+1}$, $1 \leq c+d \leq p^k$, and $a>0$.  In this $a=1$ and $c=0$.  If $|c-d|>1$, $\lambda(p^t,n,p)$ is not standard.  Assume $d=1$.  In order for $\lambda(p^t,n,p)$ to be standard in this case, $\lambda((1+b)p^k-n,(1+b) p^k-m,p)=\lambda(p^k-1,b p^k,p)$ must be standard.  But by Lemma~\ref{Lemma3}, $\lambda(p^k-1,b p^k,p)=(b p^k, \dots, b p^k)$ and so is not standard, implying that $\lambda(p^t,n,p)$ is not standard.

Case~\ref{Case4}: $m+n \leq p^{k+1}$, $1 \leq c+d \leq p^k$, $a=0$, so $m=c=p^t$, and $d>0$.  In order for $\lambda(p^t,n,p)$ to be standard in this case, $\lambda(p^t, b p^k-d,p)$ must be standard.  By since $ b p^k-d<n$, this is not standard.

Case~\ref{Case5}: $(m,n)=(p^t,b p^k)$ with $t <k$.  In this case $\lambda(m,n,p)=(n,\dots,n)$ is not standard.

Case~\ref{Case6}: $(m,n)=(p^k,bp^k)$ with $1+b\leq p$.  In this case, $\lambda(m,n,p)=(n,\dots,n)$ is not standard.
\end{proof}


\begin{proof}[Proof of Theorem~\ref{Theorem2}]
First we show that if $(m,n) \in S$, then $\lambda(m,n,p)$ is standard.  Notice that by the construction of $S$, $(m,n) \in S$ implies $m \leq n$.

By contradiction.  Let $(m,n)$ be an element of $S$ such that $\lambda(m,n,p)$ is not standard with $m+n$ as small as possible.  Then whenever $(m',n') \in S$ with $m'+n'<m+n$, $\lambda(m',n',p)$ is standard.  Hence $1 <m <p$ because $\lambda(1,n,p)$ is standard for all $n$.  Since if $(m,n) \in S_1$ then $\lambda(m,n,p)$ is standard by Lemma~\ref{Lemma1}, $m+n >p+1$ and $n=b p+d$ with $b \geq 1$ and $m-1 \leq d \leq p+1-m$.  So $d \geq 1$, $m+d \leq p+1$, and $m+p-d \leq p+1$.

Suppose that $p^k \leq n < p^{k+1}$ where $k \geq 1$, and write $n=b_1 p^k+d_1$ where $1 \leq b_1 <p$ and $0 \leq d_1 < p^k$.    Write $d_1=r p+d$, where $0 \leq r <p^{k-1}$ and $0 \leq d<p$.  So $n=b p+d$ where $b=b_1 p^{k-1}+r$.  Since $(m,n) \in S$, $m+d \leq p+1$ and $m+p-d \leq p+1$.

We now check out the relevant cases of Definition~\ref{D:one}.  The fact that $m<p \leq  p^k$ rules out Case~\ref{Case3}, and the fact that $d>0$ rules out Cases~\ref{Case5} and~\ref{Case6}.

Case~\ref{Case1}: $m+n>p^{k+1}$.  This implies $m+d>p$, and since $(m,n) \in S$, $m+d=p+1$.  Thus $m+n=p^{k+1}+1$, and $b_1=p-1$.  Then
\[(p^{k+1}-n,p^{k+1}-m)=(m-1,n-1)=(m-1,b p+d-1).\]

Note that $m-1+d-1 =p-1$ and $m-1+p-d+1=m+p-d \leq p+1$.  Thus $(m-1,b p+d-1)=(m-1,n-1) \in S$ and since $m-1+n-1 <m+n$, $\lambda(m-1,n-1,p)$ is standard, implying that $\lambda(m,n,p)$ is since by Definition~\ref{D:one}, $\lambda(m,n,p)$ consists of the top dimension $p^{k+1}$ and the $m-1$ dimensions of $\lambda(m-1,n-1,p)$.

Case~\ref{Case2}: $m+n \leq p^{k+1}$ but $m+d_1>p^k$.  This implies $m+d>p$, and so $m+d=p+1$ and $m+d_1=p^k+1$.  Then
\[((b_1+1)p^k-n,(b_1+1)p^k-m)=(m-1,n-1).\]
As in Case~\ref{Case1}, $\lambda(m-1,n-1,p)$ is standard, implying that $\lambda(m,n,p)$ is since by Definition~\ref{D:one}, $\lambda(m,n,p)$ consists of the top dimension $(b+1)p^k$ and the $m-1$ dimensions of $\lambda(m-1,n-1,p)$.

Case~\ref{Case4}: $m+n \leq p^{k+1}$, $1 \leq m+d_1 \leq p^k$. Since $(m,n) \in S$, we know $m+d \leq p+1$ and $m+p-d \leq p+1$.  Also
\[ s_p(m,b_1 p^k+d_1)_>=s_p(m,b_1 p^k-d_1)_<+2 b_1 p^k.\]
But $b_1 p^k-d_1=b_1 p^k-r p-d=(b_1 p^{k-1}-r-1)p+p-d=(b-2 r-1)p+p-d$.  If $b_1 p^k-d_1=p-d$, then $m \leq p-d$, and so $(m,b_1 p^k-d_1)\in S_1 \subset S$.  If $b_1 p^k-d_1>p-d$, then $(m,b_1 p^k-d_1)\in S_2 \subset S$.  Hence, in both cases, $(m,b_1 p^k-d_1) \in S$,  and so $\lambda(m,b_1 p^k-d_1,p)$ is standard, implying $\lambda(m,n,p)$ is as well.


Now we show that if $\lambda(m,n,p)$ is standard with $1 \leq m <p$ and $n \geq m$, then $(m,n) \in S$.  We proceed by contradiction.  Let $\lambda(m,n,p)$ be standard with $(m,n) \notin S$ such that $m+n$ is as small as possible.  Thus whenever $\lambda(m',n',p)$ is standard with $m'<p$ and $m'+n'< m+n$, $(m,n) \in S$.  Since $(1,n) \in S$ for every $n$, $m >1$.  Since an element $(k,d)$ with $1 \leq k \leq d < p$ is standard if{f} $k+d \leq p+1$, it follows that $n \geq p$, and so $n=b p+d$ with $b \geq 1$ and $0 \leq d<p$.  We now check out the relevant cases of Definition~\ref{D:one} and show $m+d \leq p+1$ and $m+ p-d \leq p+1$ in each case.

Again suppose that $p^k \leq n < p^{k+1}$ where $k \geq 1$, and write $n=b_1 p^k+d_1$ where $1 \leq b_1 <p$ and $0 \leq d_1 < p^k$.    Write $d_1=r p+d$, where $0 \leq r <p^{k-1}$ and $0 \leq d<p$.  So $n=b p+d$ where $b=b_1 p^{k-1}+r$.  Cases~\ref{Case6} and~\ref{Case3} are ruled out because $0<m<p^k$.  Since $m>1$, we rule out Case~\ref{Case5}. 

Case~\ref{Case1}: $m+n>p^{k+1}$.  Then, by Proposition~\ref{P:one}, since  $\lambda(m,n,p)$ is standard, $m+n=p^{k+1}+1$, implying $m+d=p+1$, and $\lambda(p^{k+1}-n,p^{k+1}-m,p)=\lambda(m-1,n-1,p)=\lambda(m-1,b p+d-1,p)$ is standard, hence in $(p^{k+1}-n,p^{k+1}-m)\in S$.  In particular, $m-1+p-(d-1) \leq p+1$, that is, $m+p-d \leq p+1$.  We have shown that $(m,n) \in S$.

Case~\ref{Case2}: $m+n \leq p^{k+1}$ but $m+d_1>p^k$.  Then, by Proposition~\ref{P:one}, since  $\lambda(m,n,p)$ is standard, $m+d_1=p^k+1$, implying $m+d=p+1$, and  $\lambda((b_1+1)p^k-n,(b_1+1)p^k-m,p)=\lambda(m-1,n-1,p)=\lambda(m-1,b p+d-1,p)$ is standard, hence in $((b_1+1)p^k-n,(b_1+1)p^k-m) \in S$.  Thus $m-1+p-(d-1) \leq p+1$, that is, $m+p-d \leq p+1$.  We have shown that $(m,n) \in S$.

Case~\ref{Case4}: $m+n \leq p^{k+1}$, $1 \leq m+d_1 \leq p^k$, and $d_1>0$.  Then $\lambda(m,b_1 p^k-d_1,p)=\lambda(m,(b-2 r-1)p+p-d,p)$ is standard, hence $(m,b_1 p^k-d_1) \in S$.  If $(m,b_1 p^k-d_1) \in S_2$, then $m+p-d \leq p+1$ and $m+p-(p-d)=m+d \leq p+1$.  Hence $(m,n) \in S$.  If $(m,b_1 p^k-d_1)=(m,p-d) \in S_1$, then $m \leq p-d <m+p-d \leq p+1$, and $(m,n) \in S$ in this case as well.

\end{proof}


\begin{proof}[Proof of Theorem~\ref{Theorem3}]
First we show that if $(m,n) \in S$ (with $m \leq n$), then $\lambda(m,n,p)$ is standard.

By contradiction.  Let $(m,n)$ be an element of $S$ such that $\lambda(m,n,p)$ is not standard with $m+n$ as small as possible.  Then whenever $(m',n') \in S$ with $m'+n'<m+n$, $\lambda(m',n',p)$ is standard. 

Now we show that if $n<p^{t+1}$ and $(m,n)\in S$, then $\lambda(m,n,p)$ is standard.  Since $(m,n) \in S$ and $m\leq n<p^{t+1}$, $(m,n)=(i p^t+x,j p^t+y)$  where  $1 \leq i \leq (p-1)/2$, $ i \leq j \leq p-i-1$, $x=(p^t\pm 1)/2$, and $y=(p^t\pm 1)/2$ with $x \leq y$ if $j=i$.  We go through the cases.

Case~\ref{Case1}: $m+n> p^{t+1}$.  It must be that $m+n=p^{t+1}+1$ with $i+j=p-1$, and so $x=y=(p^t+1)/2$.  Then
\[(p^{t+1}-n,p^{t+1}-m)=(i p^t+(p^t-1)/2,(p-i-1)p^t+(p^t-1)/2) \in S,\]
hence $\lambda(p^{t+1}-n,p^{t+1}-m,p)$ is standard, implying $\lambda(m,n,p)$ is standard since $\lambda(m,n,p)$ consists of the top dimension $p^{t+1}$ followed by the dimensions of $\lambda(p^{t+1}-n,p^{t+1}-m,p)$.

Case~\ref{Case2}: $m+n \leq p^{t+1}$ but $x+y >p^t$.  This implies $x=y=(p^t+1)/2$.  Then
\[((i+j+1)p^t-n,(i+j+1)p^t-m)=(i p^t+(p^t-1)/2,j p^t+(p^t-1)/2) \in S,\]
hence $\lambda((i+j+1)p^t-n,(i+j+1)p^t-m,p)$ is  standard, implying $\lambda(m,n,p)$ standard also since $\lambda(m,n,p)$ consists of the top dimension $(i+j+1)p^t$ followed by the dimensions of $\lambda((i+j+1)p^t-n,(i+j+1)p^t-m,p)$.

Case~\ref{Case3}: $1 \leq x+y \leq p^t$.  There are three cases here.  We will treat the case $x=(p^t-1)/2$ and $y=(p^t+1)/2$ --- the others are similar.  Then  
\[((i+j)p^t-n,(i+j)p^t-m)=((i-1)p^t +(p^t-1)/2,(j-1)p^t +(p^t+1)/2).\]
If $i \geq 2$, then $((i-1)p^t +(p^t-1)/2,(j-1)p^t +(p^t+1)/2) \in S$, hence $\lambda((i-1)p^t +(p^t-1)/2,(j-1)p^t +(p^t+1)/2,p)$ is standard, and so is $\lambda(m,n,p)$.  If $i=1$, then $\lambda((p^t-1)/2,(j-1)p^t +(p^t+1)/2,p)$ is standard by Lemma~\ref{Lemma2}, hence $\lambda(m,n,p)$ is. 

Cases~\ref{Case4} and \ref{Case5} do not apply because $i>0$, and Case~\ref{Case6} does not apply because $x$ and $y$ are positive.

Since we have just shown that $\lambda(m,n,p)$ is standard when $n<p^{t+1}$, we can assume that $n \geq p^{t+1}$.

Suppose that $p^k \leq n < p^{k+1}$ where $k \geq t+1$, and write $n=b_1 p^k+d_1$ where $1 \leq b_1 <p$ and $0 \leq d_1 < p^k$.    Write $d_1=r p^t+d$, where $0 \leq r <p^{k-t}$ and $0 \leq d<p^t$.  So $n=b p^t+d$ where $b=b_1 p^{k-t}+r$.  Since $(m,n) \in S$, $m= i p^t+x$ where $x=(p^t \pm 1)/2$ and $d=(p^t \pm 1)/2$.

Case~\ref{Case1}: $m+n>p^{k+1}$.  It must be that $x=d=(p^t+1)/2$ and $m+n=p^{k+1}+1$.  Then
\[(p^{k+1}-n,p^{k+1}-m)=(m-1,n-1) \in S,\]
hence $\lambda(p^{k+1}-n,p^{k+1}-m,p)$ is standard, implying $\lambda(m,n,p)$ standard.

Case~\ref{Case2}: $m+n \leq p^{k+1}$, but $m+d_1>p^k$.  Then $x+d>p^t+1$ and so $x=d=(p^t+1)/2$. Then 
\[((b_1+1)p^k-m,(b+1)p^k-m)=(m-1,n-1) \in S,\]
hence $\lambda((b_1+1)p^k-m,(b+1)p^k-m,p)$ is standard, implying $\lambda(m,n,p)$ standard.

Cases~\ref{Case3} and~\ref{Case6} do not apply.  Case~\ref{Case5} does not apply since $m \geq p^t \geq p>1$.

Case~\ref{Case4}: Here $m+d_1 \leq p^k$ with $d_1>0$.  We must show that $(m,b_1 p^k-d_1) \in S$.  Since $(m,n) \in S$, $n=\ell p^{t+1} + j p^t+y$, where $\ell \geq 0$, $i \leq j \leq p-i-1$, and $y=(p^t \pm 1)/2$.  We need to compare this representation of $n$ with $n= b_1 p^k+d_1$.  Since $\ell p^{t+1} <p^{k+1}$, $\ell< p^{k-t}$, and we can write $\ell=b p^{k-t-1}+e$ where $0<b<p$ and $0 \leq e <p^{k-t-1}$.  Therefore $n=b p^k+e p^{t+1}+j p^t+y$.  Since 
\[e p^{t+1}+j p^t+y\leq (p^{k-t-1}-1)p^{t+1}+ j p^t+y=p^k-(p^{t+1}-j p^t-y)<p^k,\]
it follows that $b=b_1$ and $d_1=e p^{t+1}+j p^t+y$.  Therefore
\begin{align*}
b_1 p^k-d_1&=b_1 p^k-e p^{t+1}-j p^t-y\\
&=b_1 p^k-(e-1)p^{t+1}+p^{t+1}-j p^t -y\\
&=(b_1 p^{k-t-1}-e-1)p^{t+1} +(p-j-1) p^t+ p^t-y.
\end{align*}

Since $b_1 p^{k-t-1}-e-1 \geq 0$ with equality if{f} $b_1=1$ and $e=p^{k-t-1}-1$, and $i \leq p-j-1 \leq p-i-1$, $(m,b_1 p^k-d_1) \in S$, hence $\lambda( m,b_1 p^k-d_1,p)$ is standard, implying that $\lambda(m,n,p)$ is standard.


Now we show that if $(m,n)$ is standard with $p^t \leq m < p^{t+1}$ and $m \leq n$, then $(m,n) \in S$.

By contradiction.  Let $(m,n)$ be a standard element not in $S$ with $m+n$ as small as possible.  Thus every standard $(m',n')$ with $m' \geq p^t$ and $m'+n'< m+n$ is an element of $S$.  Note that $m>p^t$ by Lemma~\ref{Lemma4}.

Now we show that if $n<p^{t+1}$ and $(m,n)$ is standard, then $(m,n) \in S$.  Write $m=a p^t+c$ and $n=b p^t+d$  where $1 \leq a \leq b <p$ and $0 \leq c,d <p^t$.  We go through the cases.

Case~\ref{Case1}: $m+n> p^{t+1}$.  By Proposition~\ref{P:one},  it must be that $m+n=p^{t+1}+1$ and that $\lambda(p^{t+1}-n,p^{t+1}-m,p)=\lambda(m-1,n-1,p)$ is standard, hence in S.  So $(m-1,n-1)=(i p^t +x, j p^t +y)$ where $x=(p^t \pm 1)/2$, $y=(p^t \pm 1)/2)$.  Since $m-1+n-1=p^{t+1}-1$, $x+y=p^{t}-1$.  Hence $x=y=(p^t - 1)/2$, and so $(m,n)=(i p^t +(p^t+1)/2,j p^t+(p^t+1)/2) \in S$.

Case~\ref{Case2}: $m+n \leq p^{t+1}$ but $c +d >p^t$.  By Proposition~\ref{P:one}, it must be that $c+d=p^{t}+1$ and $\lambda((a+b+1)p^t-n,(a+b+1)p^t-m)=\lambda(m-1,n-1,p)$ is standard, hence in $S$.  So $(m-1,n-1)=(i p^t +x, j p^t +y)$ where $x=(p^t \pm 1)/2$, $y=(p^t \pm 1)/2)$. Since $m-1+n-1=p^{t+1}-1$, $x+y=p^{t}-1$.  Hence $x=y=(p^t - 1)/2$, and so $(m,n)=(i p^t +(p^t+1)/2,j p^t+(p^t+1)/2) \in S$.

Case~\ref{Case3}: $ 1 \leq c+d \leq p^t$.  By Proposition~\ref{P:one}, $\lambda(\min(c,d),\max(c,d),p)$ is standard, $|c-d| \leq 1$, and $\lambda((a+b)p^t-n, (a+b)p^t-m,p)=\lambda((a-1)p^t+p^t-d,(b-1)p^t+p^t-c,p)$ is standard.  If $a \geq 2$, then $( (a-1)p^t+p^t-d,(b-1)p^t+p^t-c) \in S$.  So $p^t-d=(p^t \pm 1)/2$ and $p^t-c=(p^t \pm 1)/2$.  So either $c=d=(p^t - 1)/2$ or one of $c$ and $d$ is $(p^t - 1)/2$ while the other is $(p^t + 1)/2$.  Thus $(m,n) \in S$.  Suppose that $a=1$.  Since $\lambda(p^t-d,(b-1)p^t+p^t-c,p)$ is standard by Lemma~\ref{Lemma2}, $p^t-d+p^t-c \leq p^t+1$.  But $c+d \leq p^t$, so $p^t-d+p^t-c =p^t$ or $p^t-d+p^t-c =p^t+1$.  Now $|c-d| \leq 1$ implies either $p^t-d=p^t-c=(p^t+1)/2$ or one of $p^t-d$ and $p^t-c$ equals $(p^t-1)/2$ while the other equals $(p^t+1)/2$.  In any case, $(m,n) \in S$.

Cases~\ref{Case4},~\ref{Case5}, and~\ref{Case6} do not apply.

We can assume that $n \geq p^{t+1}$.  Suppose that $p^k \leq n < p^{k+1}$ where $k \geq t+1$, and write $n=b_1 p^k+d_1$ where $1 \leq b_1 <p$ and $0 \leq d_1 < p^k$.    Write $d_1=r p^t+d$, where $0 \leq r <p^{k-t}$ and $0 \leq d<p^t$.  So $n=b p^t+d$ where $b=b_1 p^{k-t}+r$.  Recall that $m= i p^t+x$ where $x=(p^t \pm 1)/2$.

Case~\ref{Case1}: $m+n>p^{k+1}$.  It must be than $m+n=p^{k+1}+1$, so $m+d_1=p^k+1$ and $x+d=p^t+1$.  Also $(p^{k+1}-n,p^{k+1}-m)=(m-1,n-1)$ is standard, hence in $S$.  So $(m-1,n-1)=(i p^t+(x-1), j p^t+(d-1) +\ell p^{t+1})$ where $(x-1)+(d-1)=p^t-1$.  Hence $x-1=d-1=(p^t-1)/2$ and $(m,n) \in S$.

Case~\ref{Case2}: $m+n \leq p^{k+1}$ but $m +d_1 >p^k$.  It must be that $m+d_1=p^k+1$ and $x+d=p^{t}+1$.  Also $((b_1+1) p^k-n,(b_1+1) p^k-m)=(m-1,n-1)$ is standard, hence in $S$.  So $(m-1,n-1)=(i p^t+(x-1), j p^t+(d-1) +\ell p^{t+1})$ where $(x-1)+(d-1)=p^t-1$.  Hence $x-1=d-1=(p^t-1)/2$ and $(m,n) \in S$.

Case~\ref{Case3} does not apply.

Case~\ref{Case4}: $1 \leq m+d_1 \leq p^k$ and $d_1>0$.  Then $(m, b_1 p^k-d_1)$ is standard, hence in $S$.  So
$b_1 p^k-d_1=n-2 d_1=b p^t+d-2(r p^t+d)=b p^t-2 rp^t-d=(b_1 p^{k-t}-r-1)p^t +p^t-d$ where $p^t-d=(p^t \pm 1)/2$.  Thus $d=(p^t \pm 1)/2$ and $n=(b_1 p^{k-t}+r)p^t+d$ and $(m,n) \in S$.

Cases~\ref{Case5} and~\ref{Case6} do not apply.
\end{proof}

\section{Conclusion}
We end with two questions:
\begin{enumerate}
\item What are necessary and sufficient conditions for $\lambda(m,n,2)$ to be standard?
\item In~\cite{B2011_2}, we identified generators for the cyclic modules $V_{\lambda_i}$ in terms of bases for $V_m$ and $V_n$ when $p\geq n+m-1$.  Are these still generators in all the cases when $\lambda(m,n,p)$ is standard?
\end{enumerate}


\end{document}